\documentclass[11pt,dvipsnames]{amsart}
\usepackage{amsmath,amssymb,amsfonts}
\usepackage{mathtools} 
\usepackage[mathscr]{eucal}
\usepackage[margin=1in]{geometry}
\usepackage{quiver}
\usepackage{comment}

\usepackage{tikz, tikz-cd}%
\usetikzlibrary{matrix,arrows,decorations.pathmorphing,cd,patterns,calc,backgrounds,patterns}
\tikzset{dummy/.style= {circle,fill,draw,inner sep=0pt,minimum size=1.2mm}}
\tikzset{vertex/.style={fill, circle, minimum size=.1cm, inner sep=0pt}}

\usepackage[%pagebackref,  %incompatible with bib
colorlinks, citecolor=MidnightBlue, linkcolor=black, urlcolor=MidnightBlue]{hyperref}
\hypersetup{linktoc=all,} %set to all if you want both sections and subsections linked

\usepackage[textwidth=25mm, textsize=tiny]{todonotes}
\usepackage[final]{microtype}

\usepackage{enumerate}

\usepackage{appendix}

\numberwithin{equation}{section} 
\numberwithin{figure}{section}
\usepackage[nameinlink,capitalise,noabbrev]{cleveref}

\newcommand{\newrefformat}[2]{}

\usepackage{stmaryrd}%for mapsfrom

\newcommand\restr[2]{{% we make the whole thing an ordinary symbol
  \left.\kern-\nulldelimiterspace % automatically resize the bar with \right
  #1 % the function
  \vphantom{\big|} % pretend it's a little taller at normal size
  \right|_{#2} % this is the delimiter
  }}

\usepackage[
backend=biber,
style=alphabetic,
maxbibnames=99,
url = false,
doi = false,
isbn = false,
]{biblatex}
\crefname{lemma}{Lemma}{Lemmas}
\crefname{theorem}{Theorem}{Theorems}
\crefname{definition}{Definition}{Definitions}
\crefname{proposition}{Proposition}{Propositions}
\crefname{remark}{Remark}{Remarks}
\crefname{corollary}{Corollary}{Corollaries}
\crefname{equation}{Equation}{Equations}
\crefname{construction}{Construction}{Constructions}
\crefname{ex}{Example}{Examples}
\crefname{appsec}{Appendix}{Appendices}
\crefname{subsection}{Subsection}{Subsections}

\theoremstyle{plain}
\newtheorem{theorem}[equation]{Theorem}
\newtheorem{corollary}[equation]{Corollary}
\newtheorem{proposition}[equation]{Proposition}
\newtheorem{lemma}[equation]{Lemma}

\newtheorem*{theorem*}{Theorem}

\theoremstyle{definition}
\newtheorem{definition}[equation]{Definition}
\newtheorem{example}[equation]{Example}
\newtheorem{remark}[equation]{Remark}

\newcommand{\EE}{\mathbb{E}}

\newcommand{\cO}{\mathcal{O}}

\newcommand{\cat}[1]{\mathscr{#1}}

\newcommand{\Loop}{\Omega}

\newcommand{\op}{\operatorname{op}}

\newcommand{\ob}{\operatorname{Ob}}
\renewcommand{\hom}{\operatorname{Hom}}

\newcommand{\id}{\operatorname{id}}

\newcommand{\cof}{\rightarrowtail}

\newcommand{\Sp}{\mathrm{Sp}}
\newcommand{\Wald}{\mathrm{Wald}}

\newcommand{\Perm}{\mathrm{Perm}}
\newcommand{\Sym}{\mathrm{Sym}}
\newcommand{\LaxSym}{\mathrm{LaxSym}}

\newcommand{\Null}{\mathrm{Null}}

\newcommand{\Span}{\mathrm{Span}}
\newcommand{\Fin}{\mathrm{Fin}}

\author[M. E. Calle]{Maxine E. Calle}             
\address{Department of Mathematics,
         University of Pennsylvania,
         Philadelphia, PA, 19104,
         USA}
\email{callem@sas.upenn.edu}

\author[D. Chan]{David Chan}
\address{Department of Mathematics,
         Michigan State University,
         East Lansing, MI, 48824
         USA}
\email{chandav2@msu.edu}


\usepackage[textwidth=25mm, textsize=tiny]{todonotes}
\title{Segal $K$-theory factors through Waldhausen categories}
\date{}

\begin{document}

\maketitle
\vspace{-8mm}
\begin{abstract}
    We show that Segal's $K$-theory of symmetric monoidal categories can be factored through Waldhausen categories. In particular, given a symmetric monoidal category $\cat{C}$, we produce a Waldhausen category $\Gamma(\cat{C})$ whose $K$-theory is weakly equivalent to the Segal $K$-theory of $\cat{C}$. As a consequence, we show that every connective spectrum may be obtained via Waldhausen $K$-theory.
\end{abstract}

\section{Introduction}
Pioneering work of Quillen, Segal, and Waldhausen produced three distinct approaches to higher algebraic $K$-theory \cite{bass_higher_1973,segal_categories_1974,ranicki_algebraic_1985}. Their constructions, which we denote by $K^Q$, $K^{S}$, and $K^W$, respectively, each have their own advantages and play important roles in the foundations of algebraic $K$-theory, with applications to a variety of fields including number theory, algebra, algebraic geometry, and algebraic topology.  
Given that there are several possible constructions of algebraic $K$-theory, each of which takes in a different kind of categorical input, it is convenient to know how to compare them. A theorem of Waldhausen shows that every exact category $\cat{C}$ admits a Waldhausen structure so that the resulting spectra $K^Q(\cat{C})$ and $K^W(\cat{C})$ are equivalent \cite[\S 1.9]{ranicki_algebraic_1985}, but the functors $K^S$ and $K^W$ are not so evidently comparable. While it is true that every Waldhausen category $\cat C$ has an underlying symmetric monoidal category $w\cat C$ after making a choice of coproducts, it is not true in general that $K^W(\cat{C})$ is equivalent to $K^S(w\cat{C})$. Moreover, not every symmetric monoidal category admits a Waldhausen structure; in particular, the underlying symmetric monoidal structure on a Waldhausen category is given by the categorical coproduct.

The goal of this paper is to produce an explicit comparison of Segal and Waldhausen $K$-theory, and our main theorem is the following.

\begin{theorem}\label{theorem: main theorem intro}
    There is a functor $\Gamma$ from symmetric monoidal categories to Waldhausen categories so that for any symmetric monoidal category $\cat{C}$ there is a zig-zag of natural weak equivalences $K^S(\cat{C})\xleftrightarrow{\sim} K^W(\Gamma(\cat{C}))$.
\end{theorem}

The construction of the Waldhausen category $\Gamma(\cat{C})$ is described at the beginning of \cref{section:main}. The underlying category of $\Gamma(\cat{C})$ is the Grothendieck construction on a psuedofunctor from spans of finite sets to categories that is determined by the symmetric monoidal category $\cat C$.  The Waldhausen structure on $\Gamma(\cat{C})$ is essentially lifted from one on the span category of finite sets.
%The functor $\Gamma$ is left adjoint to the forgetful functor from a certain class of Waldhausen categories to a category of symmetric monoidal categories (\cref{prop: adjunction}) and, in particular, the unit is a fully faithful inclusion $\cat C\hookrightarrow w\Gamma(\cat C)$ which induces an equivalence after $K^S$.  \dcnote{unclear if this last sentence is still true with all the unit adding nonsense}

As an application, we obtain an ``inverse $K$-theory functor'' for $K^W$, building on work of Thomason \cite{thomason_symmetric_1995} (see also \cite{mandell_inverse_2010}), who proved the corresponding result for $K^S$. 
In the following result, $\cat N(\Loop^{\infty} E)$ denotes a certain Waldhausen category of ``homotopically discrete retractive spaces'' over the infinite loop space $\Loop^{\infty}E$ (see \cref{definition: N(X)}).% and $W_K$ denotes those exact functors which induce stable equivalences of $K$-theory spectra.

\begin{theorem}\label{introthm:inverse K}
    For every connective spectrum $E$, there is a natural zig-zag of stable equivalences between $E$ and $K^W(\cat N(\Loop^{\infty} E))$. 
    %Consequently, the maps of relative categories\[
    % \begin{tikzcd}
    %     (\Wald, W_K) \ar[rr, shift right, swap, "K^W"] && (\Sp^{\geq 0}, W) \ar[ll, shift right, swap, "\cat N(\Loop^{\infty}-)"]
    % \end{tikzcd}
    % \] induce an equivalence of homotopy categories.
\end{theorem}
This theorem implies that the homotopy category of connective spectra is a retract of the localization of Waldhausen categories at those exact functors which induce stable equivalences of $K$-theory spectra. 

\begin{remark}
    Recent constructions of combinatorial $K$-theory spectra often have Waldhausen categories as special cases, particularly for \textit{squares $K$-theory} \cite{ckmz:squares, cs:2segal}. A consequence of \cref{theorem: main theorem intro} is that for any symmetric monoidal category $\cat C$ there is a natural squares structure on $\Gamma(\cat C)$, via its Waldhausen structure, so that $K^{\square}(\Gamma(\cat C))$ is equivalent to $K^S(\cat C)$. Similarly, \cref{introthm:inverse K} implies that every connective spectrum may be realized via squares $K$-theory.
\end{remark}

\subsection{Related work} 
Since Thomason's original work, the inverse $K$-theory for symmetric monoidal categories has also been considered in work of Mandell and Johnson--Yau \cite{mandell_inverse_2010,JohnsonYauInverse} and extended to multicategories in work of Johnson--Yau \cite{JohnsonYauMulti}. A version of Thomason's theorem for genuine equivariant spectra is proved by the authors and P\'eroux in \cite{CalleChanPeroux}.  Johnson--Osorno show that all $1$-truncated spectra can be modeled by Picard categories \cite{JohnsonOsorno}, and this was extended by Moser--Ozornova--Paoli--Sarazola--Verdugo to show that $n$-truncated spectra are modeled by weak $n$-groupoids \cite{MoserOzornovaPaoliSarazolaVerdugo}.  Additionally, recent work of Ramzi--Sosnilo--Winges shows that every spectrum (possibly non-connective) arises as the (non-connective) $K$-theory of a stable $\infty$-category \cite{ramzisosnilowinges}.
Since the non-connective $K$-theory of stable $\infty$-categories does not factor through Waldhausen categories, it does not seem possible to directly deduce \cref{introthm:inverse K} from this result or Thomason's original work.

\vspace{.5cm}

\subsection{Acknowledgments}
The authors thank Teena Gerhardt, Mona Merling, Shaul Ragimov, Chase Vogeli, and Lucas Williams for helpful conversations, as well as Liam Keenan, Maximilien P\'eroux, and Maru Sarazola for feedback on early drafts of the paper. Additionally, the authors are very grateful to George Raptis, Melissa Wei, and Donald Yau for insightful observations which contributed to our understanding of the category $\Gamma(\cat C)$ as a Grothendieck construction.The authors would also like to thank anonymous referee for their comments which greatly improved the paper.

The first author was partially supported by NSF grant DGE-1845298, and the second author was supported by NSF grant DMS-2135960.
The authors would also like to thank the Isaac Newton Institute for Mathematical Sciences, Cambridge, for support and hospitality during the programme ``Equivariant homotopy theory in context" where some work on this paper was undertaken. This work was supported by EPSRC grant no EP/Z000580/1.

\section{Preliminaries} 

In this section we gather together background on Segal and Waldhausen's constructions of algebraic $K$-theory.  We assume the reader is familiar with symmetric monoidal categories, and write $\Sym$ for the category of symmetric monoidal categories and strong symmetric monoidal functors.

\subsection{Segal $K$-theory}

Segal's group completion $K$-theory functor \cite{segal_categories_1974} $K^{S}\colon \Sym\to \Sp$ produces a connective spectrum from the data of a symmetric monoidal category. The underlying infinite loop space of $K^S(\cat C)$ is the group completion of the classifying space $B\cat C$, which may be modeled as $\Loop \mathbf{B}(B\cat C)$, where $\mathbf{B}(B\cat C)$ is the bar construction on the $\mathbb{E}_{\infty}$-algebra $B\cat C$. We begin with a lemma which produces zig-zags of morphisms from \emph{non-unital} strong symmetric monoidal functors.

\begin{lemma}
    A non-unital strong symmetric monoidal functor $F\colon \cat C\to \cat D$ induces a natural zigzag $K(\cat{C})\xleftarrow{\sim} \bullet\rightarrow K(\cat{D})$ where both arrows are maps of grouplike $E_{\infty}$-spaces and the left arrow is a weak equivalence.
\end{lemma}\begin{proof}
    The forgetful functor from unital to non-unital $\Bbb{E}_\infty$-spaces has a left adjoint $(-)^+$ which freely adds a unit. We may then define the group completion of a non-unital $\Bbb{E}_\infty$-space $X$ to be the group completion of $X^+$. If $X$ is already unital, then there is a natural equivalence $(X^+)^{\rm gp}\xrightarrow{\sim} X^{\rm gp}$, obtained from the forgetful-free adjunction and the universal property of group completion. Now, a non-unital symmetric monoidal functor $F\colon \cat C\to \cat D$ induces a map of \textit{non-unital} $E_\infty$-spaces $B\cat C\to B\cat D$ and therefore a map $(B\cat C^+)^{\rm gp}\to (B\cat D^+)^{\rm gp}$ on their group completions. We thus obtain a zig-zag $K(\cat C)\xleftarrow{\sim} (B\cat C^+)^{\rm gp}\to (B\cat D^+)^{\rm gp}\xrightarrow{\sim} K(\cat D)$ as claimed.
\end{proof}

We will find it convenient to use \textit{permutative categories} (also called a strict symmetric monoidal categories) rather than symmetric monoidal categories. Recall that 
a permutative category is a symmetric monoidal category for which the monoidal structure is strictly unital and strictly associative (see, e.g. \cite[Definition 4.1]{segal_e_1974} for a complete definition), and we write $\Perm$ for the category of small permutative categories and strictly unital strong symmetric monoidal functors.

% \begin{notation}
%     Let $\cat{C}$ be a permutative category.  For objects $C_1,\dots,C_n$ we write $T^n(C_1,\dots,C_n)$ for the product of the $C_i$.  Since the category is permutative, we do not need to specify an ordering of parentheses.\dcnote{do we still need this notation? I sort of stopped using it}
% \end{notation}

Every symmetric monoidal category can be \emph{strictified} to an equivalent permutative category (\cite[XI.3 Theorem 1]{MacLane}, \cite[Proposition 4.2]{segal_e_1974}). Thus, virtually any construction which can be performed on permutative categories can also be applied to an arbitrary symmetric monoidal category by first applying strictification.

\begin{definition}
    Let $\cat C$ be a permutative category. Define a new permutative category $\cat C_+$ with $\ob \cat C_+ = \ob \cat C \amalg \{*\}$ and $\hom\cat C_+ = \hom\cat C \amalg \{\id_*\}$. The permutative structure on $\cat C_+$ is given by the permutative structure on $\cat C$ but with $*$ behaving as the new (strict) unit, meaning $-\otimes * = \id = *\otimes -$.  
\end{definition}

Note that there is a homeomorphism $B(\cat C_+)\cong (B\cat C)_+$. Although the classifying spaces $B\cat{C}$ and $B(\cat{C}_+)$ are not equivalent, the group completions are.

\begin{lemma}[{\cite[Equation 1.6.3]{thomason_symmetric_1995}, see also \cite[Lemma A.1]{thomason_first_1982}}]\label{lemma: non unital equivalence}
    The (non-unital) strong symmetric monoidal inclusion $\cat C \to\cat C_+$ induces a zig-zag of weak equivalences $K^S(\cat C)\xleftrightarrow{\sim} K^S(\cat C_+)$.
\end{lemma}\begin{proof}
    %Recall that an equivalent formula for $K^S$ is via group completion, i.e. $K^S(\cat C) \simeq (B\cat C)^{\rm gp}\simeq \Loop \mathbf{B}(B\cat C)$. 
    Recall that the bar construction $\mathbf{B}(B\cat C)$ can be modeled by taking the classifying space of the one-object topological category whose morphism space is $B\cat C$; we note that if $\cat C$ is permutative, then $B\cat C$ is a strict topological monoid \cite[Theorem 4.10]{segal_e_1974} and so this indeed defines a space of morphisms. We write $*_{B\cat C}$ for this one-object topological category, to distinguish it from the space $B\cat C$. 
    Now apply \cite[Proposition 3.8]{ebert_semisimplicial_2019} to the topological category $*_{B\cat C}$, considered as non-unital (which does not change the homotopy type of the classifying space), noting that $(*_{B\cat C})^+ = *_{B\cat C_+}$. This shows that $\mathbf{B}(B\cat C)\simeq \mathbf{B}(B\cat C_+)$, and the claim follows.
\end{proof}

Although $K^S$ produces a spectrum from any symmetric monoidal category, it is typically applied to groupoids. Any symmetric monoidal category $\cat C$ has a \textit{core groupoid} $\cat C^{\cong}$, i.e. the wide subcategory of isomorphisms, and sometimes the definition of the $K$-theory of $\cat C$ is taken to be $K^S(\cat C^{\cong})$. With an eye towards Waldhausen categories, we will do something similar but allow for a larger subcategory of ``weak equivalences'' (e.g. coming from a Waldhausen structure), as in \cite{bohmann_multiplicative_2020}. 

\begin{definition}
    A \textit{symmetric monoidal category with weak equivalences} is a pair $(\cat{C},\cat{C}^{\simeq})$ where $\cat{C}$ is a symmetric monoidal category and $\cat C^{\simeq}\subseteq \cat C$ is a subcategory which contains all isomorphisms and is closed under the symmetric monoidal structure.
\end{definition}

\begin{definition}
    Let $\Sym_{we}$ denote the category whose objects are symmetric monoidal categories with weak equivalences and whose morphisms are strong symmetric monoidal functors which preserve the subcategories of weak equivalences. Similarly, we let $\Perm_{we}$ denote the full subcategory whose objects are pairs $(\cat{C},\cat{C}^{\simeq})$ where $\cat{C}$ is a permutative category.
\end{definition}

The $K$-theory of a symmetric monoidal category with weak equivalences is $K^S(\cat C, \cat C^{\simeq}) := K^S(\cat C^{\simeq})$; in particular, the underlying infinite loop space of $K^S(\cat C)$ is the group completion of the classifying space $B\cat C^{\simeq}$.  We note that most symmetric monoidal categories will not come with a natural choice of weak equivalences aside from the minimal choice $\cat C^{\simeq} = \cat C^{\cong}$ and maximal choice $\cat C^{\simeq } = \cat C$. 

We now show that every symmetric monoidal category with weak equivalences may be replaced with another which has the same $K$-theory but whose unit object is terminal.

\begin{definition}
    Let $(\cat C, \cat C^{\simeq})$ be a symmetric monoidal category with weak equivalences. Let $\cat C_*$ be the symmetric monoidal category whose objects are $\ob\cat C \amalg \{*\}$ and whose morphisms are \[\hom\cat C \amalg \{!_c\colon c\to *\mid c\in \ob \cat C\}\amalg \{\id_*\}.\] Composition works as in $\cat C$ and we declare that $c\to d\xrightarrow{!_d} *$ is $!_c$. The symmetric monoidal structure is the same as in $\cat C$, with $c\otimes * := c$ for all $c\in \ob\cat C$ and $f\otimes !_c = f$ for all $f\in \hom\cat C$; i.e. $*$ is a unit for the symmetric monoidal structure on $\cat C_*$. Set the weak equivalences of $\cat C_*$ to be $\cat C^{\simeq}\amalg \{\id_*\}$.
\end{definition}

That is, $\cat C_*$ is the category $\cat C$ with a new unit object $*$ which is terminal. Note that the original unit $\mathbb{I}$ is still a unit everywhere except $\mathbb{I}\otimes * = \mathbb{I}$.
Moreover, since $(\cat C_*)^{\simeq} = \cat C^{\simeq}_+$ by construction, the following result is immediate from \cref{lemma: non unital equivalence}. 

\begin{corollary}\label{cor:terminal unit same KT}
    The (non-unital) strong symmetric monoidal inclusion $(\cat C, \cat C^{\simeq})\hookrightarrow (\cat C_*, (\cat C_*)^{\simeq})$ induces a zig-zag of weak equivalences on $K$-theory.
\end{corollary}

In particular, taking $\cat C^{\simeq} = \cat C$, we obtain $K^S(\cat C)\simeq K^S(\cat C_+)= K^S(\cat C_*, (\cat C_*)^{\simeq})$.

\subsection{Waldhausen $K$-theory}
In \cite{ranicki_algebraic_1985}, Waldhausen develops a different $K$-theory construction that takes as input a different type of category, now called a Waldhausen category.

\begin{definition}
    A \textit{Waldhausen category} is a category $\cat C$ equipped with a zero object $*$ and two subcategories $c\cat C$ and $w\cat C$, called \textit{cofibrations} ($\cof$) and \textit{weak equivalences} ($\xrightarrow{\sim}$), satisfying the following conditions:\begin{itemize}
        \item[(i)] isomorphisms are both cofibrations and weak equivalences;
        \item[(ii)] for all objects $A\in \cat C$, the initial map $*\cof A$ is a cofibration;
        \item[(iii)] if $A\cof B$ is a cofibration and $A\to C$ is any morphism, the pushout $C\cup_A B$ exists and the induced map $C\cof C\cup_A B$ is a cofibration;
        \item[(iv)] gluing axiom: given a commutative diagram\[
        \begin{tikzcd}
            B \ar[d, swap, "\sim"] & A \ar[l,>->] \ar[r] \ar[d, "\sim"] & C \ar[d, "\sim"]\\
            B' & A' \ar[l,>->] \ar[r] & C'
        \end{tikzcd}
        \] the induced map $C\cup_A B \xrightarrow{\sim} C'\cup_{A'} B'$ is a weak equivalence.
    \end{itemize}
    An \textit{exact functor} between Waldhausen categories is a functor that preserves zero objects, cofibrations, weak equivalences, and pushouts along cofibrations. The category of (small) Waldhausen categories and exact functors is denoted $\Wald$. 
\end{definition}
Waldhausen defines a functor $K^W\colon \Wald\to \Sp$, which may be compared with Segal $K$-theory in certain cases. Every Waldhausen category has an underlying symmetric monoidal category, whose product is given by taking pushouts over the zero object. However, not every symmetric monoidal category is a Waldhausen category, as not every symmetric monoidal category has the coproduct as a choice of symmetric monoidal product.

\begin{definition}
    A \textit{Waldhausen category with choice of wedges} is the data of a Waldhausen category $\cat C$ along with, for all pairs of objects $(X,Y)\in \ob\cat{C}\times\ob\cat C$, a choice of pushout square \[
    \begin{tikzcd}
        * \ar[r, >->] \ar[d, >->] & X\ar[d, >->] \\
        Y \ar[r, >->] & X\vee Y
    \end{tikzcd}
    \] such that\[
    (X,*) \mapsto \begin{tikzcd}
        * \ar[r, >->] \ar[d, >->] & X\ar[d, >->, "\id"] \\
        * \ar[r, >->] & X
    \end{tikzcd} ~~~~\text{ and }~~~~ (*, X) \mapsto \begin{tikzcd}
        * \ar[r, >->] \ar[d, >->] & *\ar[d, >->] \\
        X \ar[r, >->, swap, "\id"] & X
    \end{tikzcd}.
    \]
    Note that the chosen maps $X\cof X\vee Y$ and $Y\cof X\vee Y$ must be cofibrations by axioms (ii) and (iii) for Waldhausen categories. A morphism of Waldhausen categories with wedges is an exact functor, and we denote the corresponding category by $\Wald_{\vee}$.
    \end{definition}

In \cite[Proposition 8.8]{bohmann_multiplicative_2020}, Bohmann--Osorno show that the there is a categorically enriched (multi)functor $\Lambda \colon \Wald_{\vee}\to \Sym_{we}$ which sends $\cat C$ to the pair $(\cat{C},w\cat{C})$, where the symmetric monoidal structure $\vee\colon \cat C\times \cat C\to \cat C$ is determined by the choices of wedges.  By postcomposing with strictification, we may freely assume that $\Lambda$ is a functor which takes values in $\Perm_{we}$.

The Waldhausen $K$-theory of $\cat C$ may not agree with the Segal $K$-theory of $\Lambda\cat{C}$. Sufficient conditions under which $K^W(\mathcal{C})\simeq K^S(\Lambda(\cat{C}))$ were given by Waldhausen \cite[\S 1.8]{ranicki_algebraic_1985} (see also \cite[Theorem 10.2]{bohmann_multiplicative_2020}).  We will use a variant of this result from a paper of Malkiewich--Merling, originally due to Badzioch--Dorabia\l a \cite{MalkiewichMerling:nonManifold}.  Recall that a Waldhausen category $\cat{C}$ satisfies the \emph{extension axiom} if whenever there is a map of cofiber sequences
\[
    \begin{tikzcd}
        A \ar[>->,r] \ar[d] & B \ar[d] \ar[->>,r] & C\ar[d]\\
        A'\ar[>->,r] & B' \ar[->>,r] & C'
    \end{tikzcd}
\]
where $A\to A'$ and $C\to C'$ are both weak equivalences then $B\to B'$ is also a weak equivalence.

\begin{definition}\label{definition: split cofibs}
    A Waldhausen category $\cat C$ \textit{has weakly split cofibrations} if for every cofibration $A\cof B$ there is a functorial choice of weak equivalence from  ${A\vee B/A}$ to  $B$, rel $A$.
\end{definition}
\begin{theorem}[{\cite[Proposition 2.16]{MalkiewichMerling:nonManifold}}]\label{theorem: weakly split K theory is the same}
    If $\cat C\in \Wald$ with choices of wedges, weakly split cofibrations, and the extension axiom, then there is a natural equivalence $K^{S}(\Lambda\cat C) \xrightarrow{\sim} K^W(\cat C)$.
\end{theorem}

\section{Proofs of main results}\label{section:main}

In this section we prove the main results of the paper.  We begin by proving \cref{theorem: main theorem intro}, which is recalled below as \cref{thm:main}.  After that we recall Thomason's work on inverse $K$-theory for symmetric monoidal categories, and combine this with \cref{theorem: main theorem intro} to deduce \cref{introthm:inverse K}.
\begin{theorem}\label{thm:main}
    There is a functor $\Gamma\colon \Perm_{we} \to \Wald$ such that for any permutative category $\cat{C}$ there is a natural zig-zag of stable equivalences $K^S(\cat{C}, \cat C^{\simeq})\xleftrightarrow{\sim} K^W(\Gamma(\cat{C}_*))$.
\end{theorem}

\begin{remark}
The natural stable equivalence above factors as $K^S(\cat{C}, \cat C^{\simeq})\to K^S(\cat C_*, (\cat C_*)^{\simeq})\leftrightarrow K^W(\Gamma(\cat C_*))$, where the first map is the equivalence of \cref{cor:terminal unit same KT}, and so in what follows we will assume we have already replaced $\cat C$ with $\cat C_*$, and for brevity just write $\cat C$ for a symmetric monoidal category whose unit object is terminal.
\end{remark}

\begin{remark}
    If $\cat{C}$ is a symmetric monoidal category and $\widehat{\cat{C}}$ is its strictification, then there is a stable equivalence $K^{S}(\cat{C})\to K^S(\widehat{\cat{C}})$.  Combining this with \cref{lemma: non unital equivalence} we see that for our purposes the case of a permutative category with terminal unit is generic, thus \cref{thm:main} recovers the statement of \cref{theorem: main theorem intro} from the introduction.
\end{remark}

The key idea in the construction of $\Gamma(\cat{C})$ is to enlarge $\cat{C}$ so that, in effect, the monoidal structure on $\cat{C}$ becomes the coproduct.  Moreover, this must be done so that $\Gamma(\cat{C})$ admits a Waldhausen structure. In the end, the category $\Gamma(\cat{C})$ we construct remains permutative, and we show there is a natural chain of stable equivalences
\[
    K^W(\Gamma(\cat{C}))\xleftarrow{\sim} K^S(w\Gamma(\cat{C})) \xleftarrow{\sim} K^S(\cat C_+) \xrightarrow{\sim} K^S(\cat{C}).
\]

\subsection{The Waldhausen category $\Gamma(\cat C)$}

We now proceed with the proof of \cref{thm:main}, beginning by defining the category $\Gamma(\cat{C})$ where $(\cat{C},\otimes,\mathbb{I})$ is a permutative category with $\mathbb{I}$ a terminal object. Let $\mathrm{Span}(\mathrm{Fin})$ denote the category of spans of finite sets.  Explicitly, the objects are the finite sets $\underline{n} = \{1,\dots,n\}$, and morphisms are isomorphism classes of spans \begin{equation}\label{equation: span 1}
        \underline{m}\xleftarrow{t} A\xrightarrow{r} \underline{n}
    \end{equation}
    where $A$ is some finite set.  Two spans are isomorphic if there is a bijection of middle sets which is compatible with the maps $r$ and $t$. Composition is given by pullback. Note that any span $\omega\colon \underline{m}\to \underline{n}$ can be ``turned around'' to produce a span $\omega^{\op}\colon \underline{n}\to \underline{m}$.

    \begin{definition}
        Let $\omega\colon \underline{m}\to \underline{n}$ be a span as in \eqref{equation: span 1}.  For any $i\in \underline{n}$, let $\omega(i)=\{t(j)\mid j\in r^{-1}(i)\}$ denote the ordered multi-subset of $\underline{m}$ (that is, $\omega(i)$ can have repeated values).  The ordering is induced by the total order of $\underline{m}$, with some (arbitrary) choice of ordering on repeated values.
    \end{definition}

    % The following lemma is an exercise in pullbacks of finite sets. We omit the proof.
    % \begin{lemma}\label{lemma: ordered multisets compose}
    %     If $\omega\colon \underline{m}\to \underline{n}$ and $\tau\colon \underline{\ell}\to \underline{m}$ are compsosable spans in $\Span(\Fin)$ then for all $1\leq i\leq n$ there is an equality of multi-sets $\bigcup\limits_{j\in \omega(i)}\tau(j) = (\tau\circ \omega)(i)$. 
    % \end{lemma}
    
    Any permutative category $(\cat{C},\otimes,\mathbb{I})$ determines a pseudofunctor $F_{\cat{C}}\colon\mathrm{Span}(\mathrm{Fin})\to \mathrm{Cat}$, which sends $\underline{n}$ to $\cat{C}^n$, and a span \eqref{equation: span 1} to a functor $\cat{C}^m\to \cat{C}^n$ which sends an $m$-tuple $\vec{y} = (y_1,\dots,y_m)$ to the $n$-tuple $\omega_*(\vec{y})$ whose $i$-th component is $\bigotimes\limits_{j\in \omega(i)} y_{j}$.

    \begin{definition}
         The category $\Gamma(\cat{C})$ is the Grothendieck construction of $F_{\cat{C}}$. The reader unfamiliar with Grothendieck constructions is recommended to \cite[Chapter 10]{JohnsonYau:2cats}.
         Objects are tuples $(x_1,\dots,x_n)\in \cat{C}^n$, and morphisms are pairs $(\omega,\vec{f})\colon (x_1,\dots,x_n)\to (y_1,\dots,y_m)$ where 
         \[
            \omega = [\underline{m}\xleftarrow{t} A\xrightarrow{r} \underline{n}]
         \]
         is a span, and $\vec{f}\colon (x_1,\dots,x_n)\to F_{\cat{C}}(\omega)(y_1,\dots,y_m)$ is a map in $\cat{C}^n$.  Explicitly, $\vec{f}$ consists of maps
         \[
            f_i\colon x_i\to \bigotimes\limits_{j\in \omega(i)} y_{j} = \omega_*(\vec{y})_i
         \]
         in $\cat{C}$. Note that, because $\cat{C}$ is a permutative category, any two choices of orderings for repeated values in $\omega(i)$ give the same value of the target. When $m=0$ we interpret empty products as the unit in $\cat{C}$, and hence $f_i$ is uniquely specified as the unit is assumed terminal in $\cat{C}$. For two composable arrows
         \[
            (x_1,\dots,x_n)\xrightarrow{(\omega,\vec{f})} (y_1,\dots,y_m)\xrightarrow{(\tau,\vec{g})}(z_1\dots,z_{\ell})
         \]
         the composite is $(\omega\circ\tau,\vec{h})$, where $h_i$ is the composite
         \[
            x_i \xrightarrow{f_i} \bigotimes\limits_{j\in \omega(i)} y_j\xrightarrow{\otimes g_j} \bigotimes\limits_{j\in \omega(i)}\bigotimes\limits_{k\in \tau(j)}z_k \cong \bigotimes_{k\in (\omega \circ \tau)(i)} z_k = (\omega\circ \tau)_*(\vec{z})_i
         \]
         where the isomorphism is a reordering isomorphism coming from the permutative structure on $\cat{C}$. We note that the reordering isomorphism only permutes $z_q \otimes z_p \mapsto z_p \otimes z_q$ when $p < q$.
\end{definition}
% \begin{remark}
%     The definition of composition can also be understood using the following reformulation.  Any span of the form $\omega = [\underline{m}\xleftarrow{t} A\xrightarrow{t} \underline{n}]$ determines, and is determined, by an $n\times m$ matrix $A^{\omega}$ with non-negative integer entries where $A^{\omega}_{i,j} = |t^{-1}(i)\cap r^{-1}(j)|$.  The value of $F_{\cat{C}}(\omega)(\vec{y})$ is obtained by treating the monoidal product as an addition and ``multiplying'' the column vector $\vec{y}$ by the matrix $A^{\omega}$.  From this perspective, the composition of spans is just matrix multiplication, and the composition in $\Gamma(\cat{C})$ is
%     \[
%         \vec{x}\xrightarrow{\vec{g}} A^{\omega} \vec{y}\xrightarrow{A^{\omega}\vec{g}} A^{\omega}A^{\tau}\vec{z} = A^{\tau\circ \omega} \vec{z}.
%     \]
% \end{remark}

\begin{example}\label{example: C into N(C)}
    Given any morphism $f\colon c\to d$ in $\cat{C}$, there is a corresponding morphism $(\id_{\underline{1}},f)\colon (c)\to (d)$ in $\Gamma(\cat{C})$ and this assignment defines a fully faithful functor $\cat{C}\to \Gamma(\cat{C})$.
\end{example}
\begin{lemma}
    The empty tuple $()$ is a zero object in $\Gamma(\cat{C})$.
\end{lemma}
\begin{proof}
    If $\vec{x} = (x_1,\dots,x_n)$ is any object in $\Gamma(\cat{C})$ there is a unique morphism $(\omega,\vec{f})\colon ()\to\vec{x}$ given by 
    \[
        \omega = [\underline{n} \leftarrow \emptyset \to \emptyset]
    \]
    and $\vec{f}$, by definition, is an empty tuple of morphisms. Similarly, there is a unique morphism $(\sigma,\vec{g})\colon (x_1,\dots,x_n)\to ()$ with $\sigma = \omega^{\op}$, and each $g_i$ is the unique map $x_i\to \mathbb{I}$.
\end{proof}

We now define the Waldhausen structure on $\Gamma(\cat{C})$ by first defining a notion of cofibration and weak equivalence in $\Span(\Fin)$.

\begin{definition}
    We say that a span $\underline{m}\xleftarrow{t} A\xrightarrow{r} \underline{n}$ is a \emph{cofibration} if $t$ is injective and $r$ is a bijection.  We say it is a \emph{weak equivalence} if $r$ is surjective and $t$ is a bijection.
\end{definition}

\begin{definition}\label{defintion: cofibs and wes}
    A morphism $(\omega,\vec{f})\colon (x_1,\dots,x_n)\to (y_1,\dots,y_m)$ in $\Gamma(\cat{C})$ is a \emph{cofibration} (resp. \emph{weak equivalence}) if each $f_i$ is an isomorphism (resp. in $\cat C^{\simeq}$), and the span $\omega\colon \underline{m}\to \underline{n}$ is a cofibration (resp. weak equivalence).
\end{definition}
Intuitively, a map $\{x_i\}_{i=1}^n \to \{y_j\}_{j=1}^m$ maps each $x_i$ to a tensor product of some multiset of $\{y_j\}_{j=1}^m$. The cofibrations map each $x_i$ to a single $y_j$ via an isomorphism in $\cat C$, with no $y_j$ repeated more than once. The weak equivalences map each $x_i$ to a tensor of a non-empty \textit{sub}set of $\{y_j\}_{j=1}^m$ via a weak equivalence in $\cat C$, and these subsets form a partition of $\{y_j\}_{j=1}^m$.

% \begin{remark}
%     The presence of $\omega^{\op}$ instead of $\omega$ in \cref{defintion: cofibs and wes} is explained by the following observation.  The Grothendieck construction $\Gamma(\cat{C})$ is naturally fibered over $\Span(\Fin)^{\op}$.  Turning spans around defines a natural duality $\Span(\Fin)^{\op}\to \Span(\Fin)$ which is the identity on objects. The condition $\omega^{\op}$ is a cofibration (weak equivalence) should be interpreted as saying the image of $(\omega,\vec{f})$ under the composite $\Gamma(\cat{C})\to \Span(\Fin)^{\op}\simeq \Span(\Fin)$ is a cofibration (weak equivalence).
% \end{remark}
 Basic properties of pullbacks of sets show that cofibrations and weak equivalences in $\Span(\Fin)$ are both closed under composition, and contain all isomorphisms. Consequently, cofibrations and weak equivalences in $\Gamma(\cat C)$ form subcategories.

\begin{example}\label{example: zero object}
    Since the unique span $\underline{n} \leftarrow \emptyset \to \emptyset$ is a cofibration in $\Span(\Fin)$, we see that the unique map $()\to (b_1,\dots,b_n)$ is always a cofibration.
\end{example}
\begin{example}\label{example: isos are cofibs and wes}
    Any isomorphism in $\Gamma(\cat{C})$ is a cofibration and a weak equivalence.
\end{example}

In proofs it will be convenient to replace arbitrary cofibrations with more standard ones.

\begin{definition}
    A \emph{standard cofibration} in $\Gamma(\cat{C})$ is a morphism \[(\omega,\vec{f})\colon (x_1,\dots,x_n)\to (x_1,\dots,x_n,z_1,\dots,z_{\ell})\] where $\omega = [\underline{n}\amalg \underline{\ell} \hookleftarrow \underline{n} \xrightarrow{=} \underline{n}]$, and the maps $f_i\colon x_i\to x_i$ are identities.  
\end{definition}

\begin{remark}\label{remark: standard cofibrations}
    Standard cofibrations are generic, in the sense that every cofibration factors as a standard cofibration followed by an isomorphism.
%for any cofibration $(\sigma,\vec{g})\colon (x_1,\dots,x_n)\to (y_1,\dots,y_m)$ there is an isomorphism $(\tau,\vec{h})\colon (y_1,\dots,y_m)\xrightarrow{\sim} (x_1,\dots,x_n,z_1,\dots,z_{m-n})$ such that $(\tau,\vec{h})\circ (\sigma,\vec{g})$ is a standard cofibration. 
Thus, when proving statements about cofibrations, it often suffices to prove the statement for standard cofibrations. We will proceed in this way, in particular, when proving the pushout and gluing axioms for $\Gamma(\cat{C})$.
\end{remark}

To see that the cofibrations and weak equivalences above actually give $\Gamma(\cat C)$ a Waldhausen structure, we need to check that $\Gamma(\cat{C})$ has pushouts along cofibrations and that the gluing axiom holds. In each case, we first check the corresponding axioms hold in $\Span(\Fin)$, as the essential ideas lift to prove the corresponding claim in $\Gamma(\cat{C})$. 

\begin{lemma}\label{lemma: pushouts in SpanFin}
    The square 
    % https://q.uiver.app/#q=WzAsOCxbMCwwLCJYIl0sWzAsMiwiWSJdLFsyLDAsIlhcXGFtYWxnIFoiXSxbMSwwLCJYIl0sWzAsMSwiQSJdLFsyLDIsIllcXGFtYWxnIFoiXSxbMSwyLCJZIl0sWzIsMSwiQVxcYW1hbGcgWiJdLFszLDAsIj0iLDJdLFszLDIsImkiLDAseyJzdHlsZSI6eyJ0YWlsIjp7Im5hbWUiOiJob29rIiwic2lkZSI6InRvcCJ9fX1dLFs0LDAsInIiXSxbNCwxLCJ0IiwyXSxbNiwxLCI9Il0sWzYsNSwiIiwyLHsic3R5bGUiOnsidGFpbCI6eyJuYW1lIjoiaG9vayIsInNpZGUiOiJ0b3AifX19XSxbNywyLCJyXFxhbWFsZyBcXGlkX1oiLDJdLFs3LDUsInRcXGFtYWxnIFxcaWQiXV0=
\[\begin{tikzcd}[row sep =scriptsize]
	X & X & {X\amalg Z} \\
	A && {A\amalg Z} \\
	Y & Y & {Y\amalg Z}
	\arrow["{=}"', from=1-2, to=1-1]
	\arrow["i", hook, from=1-2, to=1-3]
	\arrow["r", from=2-1, to=1-1]
	\arrow["t"', from=2-1, to=3-1]
	\arrow["{r\amalg \id_Z}"', from=2-3, to=1-3]
	\arrow["{t\amalg \id}", from=2-3, to=3-3]
	\arrow["{=}", from=3-2, to=3-1]
	\arrow[hook, from=3-2, to=3-3]
\end{tikzcd}\]
is a pushout in $\Span(\Fin)$.
\end{lemma}
\begin{proof}
    The square commutes in $\Span(\Fin)$, since both composites are isomorphic to the span 
    \[
        X\xleftarrow{r} A \xrightarrow{i_Y\circ t} Y\amalg Z  .
    \]
    To show the universal property of the pushout, consider a diagram
    % https://q.uiver.app/#q=WzAsMTIsWzAsMCwiWCJdLFswLDIsIlkiXSxbMiwwLCJYXFxhbWFsZyBaIl0sWzEsMCwiWCJdLFswLDEsIkEiXSxbMiwyLCJZXFxhbWFsZyBaIl0sWzEsMiwiWSJdLFsyLDEsIkFcXGFtYWxnIFoiXSxbNCw0LCJWIl0sWzAsNCwiQyJdLFs0LDAsIkJfWFxcYW1hbGcgQl9aIl0sWzMsMywiQ1xcYW1hbGcgQl9aIl0sWzMsMCwiPSIsMl0sWzMsMiwiaSIsMCx7InN0eWxlIjp7InRhaWwiOnsibmFtZSI6Imhvb2siLCJzaWRlIjoidG9wIn19fV0sWzQsMCwiciJdLFs0LDEsInQiLDJdLFs2LDEsIj0iXSxbNiw1LCIiLDIseyJzdHlsZSI6eyJ0YWlsIjp7Im5hbWUiOiJob29rIiwic2lkZSI6InRvcCJ9fX1dLFs3LDIsInJcXGFtYWxnIFxcaWRfWiIsMl0sWzcsNSwidFxcYW1hbGcgXFxpZCJdLFsxMCwyLCJwXzFcXGFtYWxnIHBfMiIsMl0sWzEwLDgsInMiXSxbOSw4LCJ1IiwyXSxbOSwxLCJxIl0sWzExLDgsInVcXGFtYWxnIHMiLDFdLFsxMSw1LCJxXFxhbWFsZyBwXzIiLDFdXQ==
    \begin{equation}\label{eq: pushout in spans}
    \begin{tikzcd}[row sep =scriptsize]
	X & X & {X\amalg Z} &&  \\
	A && {A\amalg Z} & {B_X\amalg B_Z} \\
	Y & Y & {Y\amalg Z} \\
	& C && \textcolor{gray}{C\amalg B_Z} \\
	 &&&& V
	\arrow["{=}"', from=1-2, to=1-1]
	\arrow["i", hook, from=1-2, to=1-3]
	\arrow["{p_1\amalg p_2}"', bend right=10, from=2-4, to=1-3]
	\arrow["s", bend left=10, from=2-4, to=5-5]
	\arrow["r", from=2-1, to=1-1]
	\arrow["t"', from=2-1, to=3-1]
	\arrow["{r\amalg \id_Z}"', from=2-3, to=1-3]
	\arrow["{t\amalg \id}", from=2-3, to=3-3]
	\arrow["{=}", from=3-2, to=3-1]
	\arrow[hook, from=3-2, to=3-3]
	\arrow[gray, dashed, "\textcolor{gray}{q\amalg p_2}"{description}, from=4-4, to=3-3]
	\arrow[gray, dashed, "\textcolor{gray}{u\amalg s}"{description}, from=4-4, to=5-5]
	\arrow["q", bend left=10, from=4-2, to=3-1]
	\arrow["u"', bend right=10, from=4-2, to=5-5]
\end{tikzcd}
\end{equation}
where the outside commutes in $\Span(\Fin)$.  The fact that the outside of this square commutes forces a bijection $B_X\cong C\times_Y A$, as sets over $X$ and $V$, and with this one can check that the entire diagram commutes in $\Span(\Fin)$.
Next, to see that the span with middle-set $C\amalg B_Z$ is the unique arrow which makes this diagram commute, note that any span $Y\amalg Z\xleftarrow{f} D\xrightarrow{g} V$ can be decomposed as 
\[
    Y\amalg Z\xleftarrow{f_1\amalg f_2} D_1\amalg D_2\xrightarrow{g_1\amalg g_2} V
\]
by letting $D_1 = f^{-1}(Y)$ and $D_2 = f^{-1}(Z)$.  The top-right region of the diagram forces $D_2\cong B_Z$ over $Z$ and $V$, and the bottom-left region forces $D_1\cong C$ over $Y$ and $V$, which completes the proof. 
\end{proof}
\begin{corollary}\label{cor: pushouts along cofibs}
    The category $\Gamma(\cat{C})$ has pushouts along cofibrations.
\end{corollary}
\begin{proof}
Suppose we are given a span\[
(y_1, \dots, y_m) \xleftarrow{(\omega, \vec{f})} (x_1, \dots, x_n) \xrightarrow{(\theta, \id_{x})} (x_1,\dots, x_n, z_1, \dots, z_l)
\] in $\Gamma(\cat C)$, where the right arrow is a standard cofibration, with $\theta  = [ \underline{n}\amalg \underline{\ell} \hookleftarrow \underline{n}\xrightarrow{=} \underline{n}] $, and write $\omega   = [ \underline{m} \xleftarrow{r} A\xrightarrow{t} \underline{n}]$. We claim that the pushout is given by\[\begin{tikzcd}[row sep =scriptsize]
	{(x_1,\dots,x_n)} &&& {(x_1,\dots,x_n,z_1,\dots z_\ell)} \\
	{(y_1,\dots,y_m)} &&& {(y_1,\dots,y_m,z_1,\dots,z_{\ell})} \\
	\arrow["{(\theta,\id_{x}) }", from=1-1, to=1-4]
	\arrow["{(\omega,\vec{f})}"', from=1-1, to=2-1]
	\arrow["{(\sigma,\vec{h})}", from=1-4, to=2-4]
	\arrow["{(\phi,\id_{y})}"', from=2-1, to=2-4]
\end{tikzcd}\]
where $ \phi = [ \underline{m}\amalg \underline{\ell} \hookleftarrow \underline{m}\xrightarrow{=} \underline{m}]$ and $ \sigma = [ \underline{m}\amalg \underline{\ell} \xleftarrow{r\amalg \id_{\underline{\ell}}} A\amalg\underline{l}\xrightarrow{t\amalg \id_{\underline{\ell}}} \underline{n}\amalg \underline{\ell}]$. We define $\vec{h}$ by
\[
    h_i = \begin{cases}
        f_i\colon x_i\to \bigotimes\limits_{j\in \omega(i)} y_j& i\leq m\\
        \id_{z_i}\colon z_i\to z_i & i>m
    \end{cases}
\]
and it is straightforward to check that the square commutes, and we note that $(\phi, \id_y)$ is a standard cofibration by construction. To show this square is a pushout, suppose we have a diagram in $\Gamma(\cat{C})$ 
    % https://q.uiver.app/#q=WzAsNSxbMCwwLCIoeF8xLFxcZG90cyx4X24pIl0sWzMsMCwiKHhfMSxcXGRvdHMseF9uLHpfMSxcXGRvdHMgel9cXGVsbCkiXSxbMCwxLCIoeV8xLFxcZG90cyx5X20pIl0sWzMsMSwiKHlfMSxcXGRvdHMseV9tLHpfMSxcXGRvdHMsel97XFxlbGx9KSJdLFs0LDIsIih2XzEsXFxkb3RzLHZfe2t9KSJdLFswLDIsIihcXG9tZWdhLFxcdmVje2Z9KSIsMl0sWzAsMSwiKFxcdGhldGEsXFxpZF97eH0pICJdLFsxLDMsIihcXHNpZ21hLFxcdmVje2h9KSJdLFsyLDMsIihcXHBoaSxcXGlkX3t5fSkiLDJdLFsyLDQsIihcXGFscGhhLFxcdmVje2N9KSIsMix7ImN1cnZlIjozfV0sWzEsNCwiKFxcYmV0YSxcXHZlY3tifSkiLDAseyJjdXJ2ZSI6LTR9XSxbMyw0LCIoXFxtdSxcXHZlY3tnfSkiLDEseyJzdHlsZSI6eyJib2R5Ijp7Im5hbWUiOiJkYXNoZWQifX19XV0=
\[\begin{tikzcd}[row sep =scriptsize]
	{(x_1,\dots,x_n)} &&& {(x_1,\dots,x_n,z_1,\dots z_\ell)} \\
	{(y_1,\dots,y_m)} &&& {(y_1,\dots,y_m,z_1,\dots,z_{\ell})} \\
	&&&& {(v_1,\dots,v_{k})}
	\arrow["{(\theta,\id_{x}) }", from=1-1, to=1-4]
	\arrow["{(\omega,\vec{f})}"', from=1-1, to=2-1]
	\arrow["{(\sigma,\vec{h})}", from=1-4, to=2-4]
	\arrow["{(\beta,\vec{b})}", curve={height=-24pt}, from=1-4, to=3-5]
	\arrow["{(\phi,\id_{y})}"', from=2-1, to=2-4]
	\arrow["{(\alpha,\vec{c})}"', curve={height=18pt}, from=2-1, to=3-5]
	\arrow[gray, "\textcolor{gray}{(\mu,\vec{g})}"{description}, dashed, from=2-4, to=3-5]
\end{tikzcd}\]
so that the outside of the diagram commutes; our goal is to construct $(\mu, \vec{g})$. We write
\[
\alpha   = [ \underline{k} \xleftarrow{u} C\xrightarrow{q} \underline{m}] \text{ and } \beta   = [ \underline{k} \xleftarrow{s} B_X\amalg B_Z\xrightarrow{p_1\amalg p_2} \underline{n}\amalg \underline{\ell}]
\]
% \begin{eqnarray*}
%     \theta =  [\underline{n}\amalg \underline{\ell} \hookleftarrow \underline{n}\xrightarrow{=} \underline{n}]  & \omega   = [ \underline{m} \xleftarrow{r} A\xrightarrow{t} \underline{n}]\\
%      \phi   =  [\underline{m}\amalg \underline{\ell} \hookleftarrow \underline{m}\xrightarrow{=} \underline{m}] & \alpha   = [ \underline{k} \xleftarrow{u} C\xrightarrow{q} \underline{m}]\\
%     \sigma   = [ \underline{m}\amalg \underline{\ell} \xleftarrow{r\amalg \id_{\underline{\ell}}} A\amalg\underline{l}\xrightarrow{t\amalg \id_{\underline{\ell}}} \underline{n}\amalg \underline{\ell}] & \beta   = [ \underline{k} \xleftarrow{s} B_X\amalg B_Z\xrightarrow{p_1\amalg p_2} \underline{n}\amalg \underline{\ell}].
% \end{eqnarray*}
so that the middle sets line up with the middle sets in the diagram \eqref{eq: pushout in spans}.  In particular, the span $\mu$ is uniquely determined as
\[
    \mu = [\underline{k} \xleftarrow{u\amalg s}C\amalg B_Z \xrightarrow{q\amalg p_2} \underline{m}\amalg \underline{\ell}].
\]
 What remains is to construct $\vec{g}$ so that the the diagram commutes in $\Gamma(\cat{C})$, and show this choice of $\vec{g}$ is unique. Note that for $1\leq i \leq m$ we have $\mu(i) = \alpha(i)$ and for $1\leq k\leq \ell$ we have $\mu(m+k) = \beta(n+k)$. Therefore, for $1\leq i\leq m$, we need to define a map
\[
    g_i\colon  y_i\to \bigotimes_{j\in \mu(i)}v_{j} =\bigotimes_{j\in \alpha(i)}v_{j} ;
\] and commutativity of the bottom triangle necessitates that the composite
\[
    y_i \xrightarrow{id_{y_i}} y_i\xrightarrow{g_i}\bigotimes_{j\in \alpha(i)}v_{j} 
\]
be equal to $c_i$, so we can (and must) take $g_i= c_i$. Similarly, for $1\leq k\leq \ell$, we need to define
\[
    g_{k+m}\colon  z_k\to \bigotimes_{j\in \mu(m+k)}w_{j} = \bigotimes_{j\in \beta(n+k)}w_{j} 
\]
and commutativity of the right-triangle implies that $g_{k+m}=b_{k+m}$. Since the values of $g_i$ are forced, the vector $\vec{g}$ is unique, so the top left-square is a pushout in $\Gamma(\cat{C})$. 
\end{proof}

\begin{lemma}
  Given a diagram in $\Span(\Fin)$
    % https://q.uiver.app/#q=WzAsMjAsWzYsMCwiWCJdLFs4LDAsIlkiXSxbNCwwLCJYXFxhbWFsZyBaIl0sWzQsMiwiWCdcXGFtYWxnIFonIl0sWzYsMiwiWCciXSxbOCwyLCJZJyJdLFs0LDEsIlgnXFxhbWFsZyBaJyJdLFs2LDEsIlgnIl0sWzgsMSwiWSciXSxbMSwwLCJcXGJ1bGxldCJdLFswLDAsIlxcYnVsbGV0Il0sWzAsMiwiXFxidWxsZXQiXSxbMSwyLCJcXGJ1bGxldCJdLFsyLDAsIlxcYnVsbGV0Il0sWzIsMiwiXFxidWxsZXQiXSxbMywxLCI9Il0sWzUsMiwiWCciXSxbNSwwLCJYIl0sWzcsMCwiQSJdLFs3LDIsIkEnIl0sWzgsMSwicCIsMix7InN0eWxlIjp7ImhlYWQiOnsibmFtZSI6ImVwaSJ9fX1dLFs3LDAsInIiLDAseyJzdHlsZSI6eyJoZWFkIjp7Im5hbWUiOiJlcGkifX19XSxbNyw0LCI9Il0sWzgsNSwiPSJdLFs5LDEwXSxbOSwxM10sWzEwLDExXSxbOSwxMl0sWzEyLDExXSxbMTIsMTRdLFsxMywxNF0sWzYsMiwiclxcYW1hbGcgcyIsMCx7InN0eWxlIjp7ImhlYWQiOnsibmFtZSI6ImVwaSJ9fX1dLFs2LDMsIj0iLDJdLFsxNiwzLCIiLDEseyJzdHlsZSI6eyJ0YWlsIjp7Im5hbWUiOiJob29rIiwic2lkZSI6ImJvdHRvbSJ9fX1dLFsxNywyLCIiLDAseyJzdHlsZSI6eyJ0YWlsIjp7Im5hbWUiOiJob29rIiwic2lkZSI6ImJvdHRvbSJ9fX1dLFsxNywwLCI9Il0sWzE2LDQsIj0iLDJdLFsxOSw0XSxbMTksNV0sWzE4LDBdLFsxOCwxXV0=
\[\begin{tikzcd}[row sep = scriptsize]
	X\amalg Z & X & Y && {X\amalg Z} & X & X & A & Y \\
	&&& {=} & {X'\amalg Z'} && {X'} && {Y'} \\
	X' \amalg Z' & X' & Y' && {X'\amalg Z'} & {X'} & {X'} & {A'} & {Y'}
	\arrow[from=1-1, to=3-1, swap, "\sim"]
	\arrow[from=1-2, to=1-1]
	\arrow[from=1-2, to=1-3]
	\arrow[from=1-2, to=3-2, "\sim"]
	\arrow[from=1-3, to=3-3, "\sim"]
	\arrow[hook', from=1-6, to=1-5]
	\arrow["{=}", from=1-6, to=1-7]
	\arrow[from=1-8, to=1-7]
	\arrow[from=1-8, to=1-9]
	\arrow["{r\amalg s}", two heads, from=2-5, to=1-5]
	\arrow["{=}"', from=2-5, to=3-5]
	\arrow["r", two heads, from=2-7, to=1-7]
	\arrow["{=}", from=2-7, to=3-7]
	\arrow["p"', two heads, from=2-9, to=1-9]
	\arrow["{=}", from=2-9, to=3-9]
	\arrow[from=3-2, to=3-1]
	\arrow[from=3-2, to=3-3]
	\arrow[hook', from=3-6, to=3-5]
	\arrow["{=}"', from=3-6, to=3-7]
	\arrow[from=3-8, to=3-7]
	\arrow[from=3-8, to=3-9]
\end{tikzcd},\]
the induced map on pushouts is the opposite span of a weak equivalence in $\Span(\Fin)$.
\end{lemma}
\begin{proof}
    By the construction of the universal map in diagram \eqref{eq: pushout in spans}, the induced map is
    \[
        Y\amalg Z \xleftarrow{p\amalg s}Y'\amalg Z' \xrightarrow{=} Y'\amalg Z'
    \]
    which is the opposite of a weak equivalence since both $p$ and $s$ are surjective by assumption.
\end{proof}

\begin{corollary}\label{cor: gluing axiom}
    The gluing axiom holds in $\Gamma(\cat{C})$.
\end{corollary}
\begin{proof}
    Consider a commutative diagram
    % https://q.uiver.app/#q=WzAsNixbMiwwLCIoeF8xLFxcZG90cyx4X24pIl0sWzAsMCwiKHhfMSxcXGRvdHMseF9uLHpfMSxcXGRvdHMsel97XFxlbGx9KSJdLFs0LDAsIih5XzEsXFxkb3RzLHlfbSkiXSxbMiwyLCIoeCdfMSxcXGRvdHMseCdfTikiXSxbMCwyLCIoeF8xJyxcXGRvdHMseF9OJyx6XzEnLFxcZG90cyx6X3tMfScpIl0sWzQsMiwiKHlfMScsXFxkb3RzLHlfTScpIl0sWzEsNCwiKFxcb21lZ2EsXFx2ZWN7Zn0pIl0sWzAsMywiKFxcc2lnbWEsXFx2ZWN7ZX0pIl0sWzIsNSwiKFxcdGF1LFxcdmVje2h9KSJdLFszLDUsIihcXHBzaSxcXHZlY3tjfSkiLDJdLFszLDQsIihcXGdhbW1hLFxcdmVje2R9KSJdLFswLDEsIihcXHRoZXRhLFxcdmVje2F9KSIsMl0sWzAsMiwiKFxccGhpLFxcdmVje2J9KSJdXQ==
\[\begin{tikzcd}[row sep =small]
	{(x_1,\dots,x_n,z_1,\dots,z_{\ell})} && {(x_1,\dots,x_n)} && {(y_1,\dots,y_m)} \\
	\\
	{(x_1',\dots,x_N',z_1',\dots,z_{L}')} && {(x'_1,\dots,x'_N)} && {(y_1',\dots,y_M')}
	\arrow["{(\omega,\vec{f})}", swap, from=1-1, to=3-1]
	\arrow["{(\theta,\vec{a})}"', >->, from=1-3, to=1-1]
	\arrow["{(\phi,\vec{b})}", from=1-3, to=1-5]
	\arrow["{(\sigma,\vec{e})}", from=1-3, to=3-3]
	\arrow["{(\tau,\vec{h})}", from=1-5, to=3-5]
	\arrow["{(\gamma,\vec{d})}", >->, from=3-3, to=3-1]
	\arrow["{(\psi,\vec{c})}"', from=3-3, to=3-5]
\end{tikzcd}\]
in $\Gamma(\cat{C})$ where the vertical maps are weak equivalences and the horizontal maps in the left-square are cofibrations.  Generically, the spans have the form
\begin{align*}
    \tau &= [\underline{M}\xleftarrow{=} \underline{M}\xrightarrow{p} \underline{m}]
      &  \theta &= [\underline{n}\amalg \underline{\ell}\hookleftarrow \underline{n}\xrightarrow{=} \underline{n}]
   &  \phi &= [\underline{m}\xleftarrow{f}A\xrightarrow{g} \underline{n}]\\
     \sigma &= [\underline{N}\xleftarrow{=} \underline{N}\xrightarrow{r} \underline{n}]
     & \gamma &= [\underline{N}\amalg \underline{L}\hookleftarrow \underline{N}\xrightarrow{=} \underline{N}] & \psi  &= [\underline{M}\xleftarrow{f'}A'\xrightarrow{g'} \underline{N}]\\
     \omega & = [\underline{N}\amalg \underline{L}\xleftarrow{=} \underline{N}\amalg \underline{L}\xrightarrow{r\amalg s} \underline{n}\amalg\underline{\ell}] & & & & 
\end{align*}
where $p$, $r$ and $s$ are all surjective maps. By the proof of \cref{cor: pushouts along cofibs}, the induced map on pushouts for this diagram is 
\[
    (\mu,\vec{g})\colon (y_1,\dots,y_m,z_1,\dots,z_{\ell})\to (y_1',\dots,y_{M}',z_1',\dots,z_{L}')
\]
where $\mu = [\underline{M}\amalg \underline{L}\xleftarrow{=} \underline{M}\amalg \underline{L} \xrightarrow{p\amalg s} \underline{m}\amalg \underline{\ell}]$ and every component of $\vec{g}$ is either an $f_i$ or an $h_i$.  Thus $\mu$ is a weak equivalence in $\Span(\Fin)$ and since all the vertical maps in the diagram are weak equivalences, all the $h_i$ and $f_i$ are weak equivalences in $\cat{C}$, so all the $g_i$ are weak equivalences in $\cat{C}$. This shows $(\mu,\vec{g})$ is a weak equivalence in $\Gamma(\cat{C})$.
\end{proof}

\begin{theorem}\label{theorem: Gamma is Waldhausen}
    The cofibrations and weak equivalences of \cref{defintion: cofibs and wes} give $\Gamma(\cat{C})$ a Waldhausen structure.  This construction is functorial in strong symmetric monoidal functors $(\cat{C},\cat{C}^{\simeq})\to (\cat{D},\cat{D}^{\simeq})$.
\end{theorem}
\begin{proof}
    The first statement is the content of \cref{cor: gluing axiom,cor: pushouts along cofibs,example: zero object,example: isos are cofibs and wes}.  For functoriality, note that a strong symmetric monoidal functor $\Phi\colon(\cat{C},\cat{C}^{\simeq})\to (\cat{D},\cat{D}^{\simeq})$ determines a pseudonatural transformation between pseudofunctors $F_{\cat{C}}\Rightarrow F_{\cat{D}}$ which, in turn, determines a functor $\Gamma(\Phi)\colon\Gamma(\cat{C})\to \Gamma(\cat{D})$ on Grothendieck constructions \cite[Lemma 10.3.5]{JohnsonYau:2cats}. Moreover, $\Gamma(\Phi)$ is exact because on morphisms $\Gamma(\Phi)$ is the identity on the span component (and $\Phi(\cat{C}^{\simeq})\subset \cat{D}^{\simeq}$ by assumption).  Checking that $\Gamma(\Phi)$ preserves pushouts along cofibrations is straightforward from the explicit description given in the proof of \cref{cor: pushouts along cofibs}.
\end{proof}

We now show that the cofibrations in $\Gamma(\cat{C})$ are weakly split, in the sense of \cref{definition: split cofibs}.
If $(\omega,f)\colon X\to Y$ is a cofibration in $\Gamma(\cat{C})$, we write $Y/X$ for the cofiber.  
\begin{proposition}\label{prop:weakly split cof}
    Given a cofibration $(\omega,\vec{f})\colon X\to Y$ in $\Gamma(\cat{C})$, there is a natural isomorphism $(\sigma,\vec{g})\colon X\vee (Y/X)\to Y$.
\end{proposition}
\begin{proof}
    Since standard cofibrations are generic, it suffices to prove the claim in the case  
    \[
        (\omega,\id_x)\colon (x_1,\dots,x_n)\to (x_1,\dots,x_n,z_{1},\dots,z_{\ell})
    \]
    where $\omega = [\underline{n}\amalg \underline{\ell}\hookleftarrow \underline{n} \xrightarrow{=} \underline{n}]$. In this case, the cofiber $Y/X = (z_1,\dots,z_{\ell})$ and the desired natural weak equivalence is the identity; more generally it is the canonical isomorphism which identifies $Y$ as a coproduct of $X$ and $Y/X$.
\end{proof}
\begin{remark}\label{remark: extension axiom}
    Because cofibration sequences in $\Gamma(\cat{C})$ are split up to natural \emph{isomorphism}, and weak equivalences are closed under coproducts, we observe that $\Gamma(\cat{C})$ satisfies the extension axiom.
\end{remark}

\subsection{Comparing the $K$-theory of $\Gamma(\cat C)$ and $\cat C$}
Finally, we prove that there is a natural weak equivalence $K^S(\cat C, \cat C^{\simeq})\xrightarrow{\sim} K^S(\Gamma(\cat{C}),w\Gamma(\cat C))$. This will be realized as the composite of maps  $K^S(\cat C, \cat C^{\simeq})\to K^S(\cat C_+, \cat C^{\simeq}_+)\to K^S(\Gamma(\cat{C}),w\Gamma(\cat C))$, where the second map is induced by the inclusion $\cat C_+\hookrightarrow \Gamma(\cat C)$ which sends $c\in \cat C$ to the singleton $(c)$ and $+\mapsto ()$. In the following proof, it becomes clear that we need to add in $+$ so that something maps to $()$.

\begin{proposition}\label{proposition: Quillen theorem A}
    The inclusion $s\colon\cat{C}^{\simeq}_+\to w\Gamma(\cat{C})$ is strictly unital oplax symmetric monoidal and induces a homotopy equivalence on classifying spaces.
\end{proposition}
\begin{proof}
    The functor $s$ is strictly unital by construction. To see $s$ is oplax, define structure maps \begin{equation}\label{equation: terminal object}
        (\omega,f)\colon(c_1\otimes\dots\otimes c_n)\to (c_1,\dots,c_n)
    \end{equation}
    in $w\Gamma(\cat{C})$ where $\omega = [\underline{n}\xleftarrow{=}\underline{n} \to \underline{1}]$ and $f = \id_{c_1\otimes\dots\otimes c_n}$. Since $\omega^{\op}$ is a weak equivalence in $\Span(\Fin)$, and $f$ is an isomorphism, $(\omega,f)$ is indeed a morphism in $w\Gamma(\cat{C})$. That this defines an oplax structure map follows from the composition of spans and the fact that identities compose to identities.

    To see that $s$ induces an equivalence on classifying spaces we apply Quillen Theorem A.  That is, we need to show, for any object $\vec{c}\in \Gamma(\cat{C})$, that the fiber category $s\downarrow\vec{c}$ is contractible. 
    
    The fiber category $s\downarrow ()$ has only one object, the identity $s(+)=()$, as there are no other weak equvialences to $()$ in $\Gamma(\cat{C})$. For $(c_1,\dots,c_n)$ with $n\geq 1$, we will prove that the category $s\downarrow (c_1,\dots,c_n)$ is contractible by showing it has a terminal object given by
    \[
        s(c_1\otimes\dots\otimes c_n)\xrightarrow{(\phi,f)} (c_1,\dots,c_n)
    \]
    where $(\phi,f)$ is the same map as \eqref{equation: terminal object}.
    % To see this is a terminal object, suppose we are given a weak equivalence $(\psi,g)\colon s(X)\to (c_1,\dots,c_n)$, we want to produce a unique $(\alpha,h)$ making the diagram
    % % https://q.uiver.app/#q=WzAsMyxbMSwwLCIoY18xXFxvdGltZXNcXGRvdHNcXG90aW1lcyBjX3ApIl0sWzAsMCwiKFgpIl0sWzEsMSwiKGNfMSxcXGRvdHMsY197cH0pIl0sWzEsMCwiKFxcYWxwaGEsaCkiXSxbMCwyLCIoXFxwaGksZikiXSxbMSwyLCIoXFxwc2ksZykiLDJdXQ==
    % \[\begin{tikzcd}
    % 	{s(X)} & {s(c_1\otimes\dots\otimes c_p)} \\
    % 	& {(c_1,\dots,c_{p})}
    % 	\arrow["{(\alpha,h)}", from=1-1, to=1-2]
    % 	\arrow["{(\psi,g)}"', from=1-1, to=2-2]
    % 	\arrow["{(\phi,f)}", from=1-2, to=2-2]
    % \end{tikzcd}\]
    % commute. Since the span $\psi\colon\underline{n}\to \underline{1}$ must be the opposite of a weak equivalence, we must have $\psi = \phi$ and so we must have $\alpha = \id_{\underline{1}}$. Moreover, since the map $f$ is an identity, the only possible choice is $h = g$, and thus the choice $(\alpha,h)$ is unique, as desired.
    Given a weak equivalence $(\psi,g)\colon s(X)\to (c_1,\dots,c_n)$, we want to show there is a unique weak equivalence $h\colon X\to c_1\otimes\dots\otimes c_n$ in $\cat{C}_+$ making the diagram
    % https://q.uiver.app/#q=WzAsMyxbMSwwLCIoY18xXFxvdGltZXNcXGRvdHNcXG90aW1lcyBjX24pIl0sWzAsMCwiKFgpIl0sWzEsMSwiKGNfMSxcXGRvdHMsY197bn0pIl0sWzEsMCwicyhoKSJdLFswLDIsIihcXHBoaSxmKSJdLFsxLDIsIihcXHBzaSxnKSIsMl1d
\[\begin{tikzcd}
	{(X)} & {(c_1\otimes\dots\otimes c_n)} \\
	& {(c_1,\dots,c_{n})}
	\arrow["{s(h)}", from=1-1, to=1-2]
	\arrow["{(\psi,g)}"', from=1-1, to=2-2]
	\arrow["{(\phi,f)}", from=1-2, to=2-2]
\end{tikzcd}\]
    commute.  By definition, $s(h) = (\id_{\underline{1}},h)$, and the composite $(\phi,f)\circ (s,h) = (\phi,h)\colon (X)\to (c_1,\dots,c_n)$. Since $(\psi,g)$ is a weak equivalence we must have $\psi = [\underline{n}\xleftarrow{=}\underline{n}\to \underline{1}] = \phi$, and thus we have $(\phi,f)\circ (s,h) = (\phi,h) = (\psi,g)$ if and only if $h = g$.  In particular $h=g$ defines the unique morphism from $s(X)$ to $(c_1\otimes\dots\otimes c_n)$ in the over category $s\downarrow(c_1,\dots,c_n)$, hence $(c_1\otimes\dots\otimes c_n)$ is a terminal element in this category.
    
\end{proof}

The strictly unital oplax symmetric monoidal functor $s$ in the proposition above is sent, by the oplax Segal construction (see \cite[Section 8.3]{JohnsonYau:Bimonoidal}), to a morphism of special $\Gamma$-categories that is a nerve-equivalence at level 1, by \cref{proposition: Quillen theorem A}, and hence induces a level weak equivalence (hence stable equivalence) of Segal $K$-theory spectra. We have now assembled everything we need to prove the main theorem.  
% We note that an oplax symmetric monoidal functor induces an $\EE_{\infty}$-map on classifying spaces, and therefore a map on group completions. To see this, observe that if $F$ is oplax then the opposite map $F^{\op}\colon \cat{C}^{\op}\to \cat{D}^{\op}$ is lax monoidal and the composite   \[
%         B\cat{C}\cong B\cat{C}^{\op}\xrightarrow{BF^{\op}}B\cat{D}^{\op}\cong B\cat{D}
% \] is identical to the map $F$ and is a composite of $\EE_{\infty}$-maps. 

\begin{proof}[Proof of \cref{thm:main}]
    By \cref{prop:weakly split cof} and \cref{theorem: Gamma is Waldhausen}, $\Gamma(\cat{C})$ is a Waldhausen category with choices of wedges and weakly split cofibrations, and $\Gamma(\cat{C})$ satisfies the extension axiom by \cref{remark: extension axiom}. Thus by \cref{theorem: weakly split K theory is the same} we have a weak equivalence $K^S(\Lambda\Gamma(\cat{C}))\to K^W(\Gamma(\cat{C}))$. \cref{proposition: Quillen theorem A,lemma: non unital equivalence} combine to give a map of $\EE_{\infty}$-spaces $B\cat{C}^{\simeq}\to Bw\Gamma(\cat{C})$ which is a weak equivalence and so group completes to a weak equivalence. There is thus a zig-zag of equivalences 
    \[
        K^S(\cat{C}, \cat C^{\simeq})\xleftarrow{\sim} K^S(\cat{C}_+, \cat C^{\simeq}_+)\xrightarrow{\sim} K^S(\Lambda\Gamma(\cat{C}))\xrightarrow{\sim} K^W(\Gamma(\cat{C})).
    \]
    \end{proof}

\subsection{Inverse Waldhausen $K$-theory}
\cref{thm:main} quickly implies one version of \cref{introthm:inverse K}.

\begin{theorem}\label{thm: inverse wald version 1}
    For every connective spectrum $E$, there is a Waldhausen category $\cat C$ and a zig-zag of weak equivalences $K^W(\cat C)\leftrightarrow E$. 
\end{theorem}
\begin{proof}
    In \cite[\S 5]{thomason_symmetric_1995}, Thomason shows that there is a functor $\Sp^{\geq 0}\to \Sym$ which is an inverse equivalence, after localization at stable equivalences, to $K^S$. Strictifying, and adding a terminal unit, we may assume this functor lands in the full subcateogry of $\Perm_{we}$ on objects with terminal units. Post-composing with the functor $\Gamma\colon \Sym\to \Wald_{\vee}$, we obtain a functor $\Sp^{\geq 0}\to \Wald_{\vee}$, and by \cref{thm:main} this defines a homotopy right-inverse to $K^W$.
\end{proof}
The Waldhausen category $\cat C$ in the theorem above is obtained by applying $\Gamma$ to the output of Thomason's inverse $K$-theory functor on $E$. 
In the remainder of the paper, we give an explicit description of a Waldhausen category $\cat{N}(X)$, for $X$ a connective spectrum, such that $K^W(\cat{N}(X))\simeq X$. Rather than unpacking the result of the proof of \cref{thm: inverse wald version 1}, we outline a slight modification of the $\Gamma$ construction which skips the strictification procedure, as this leaves the underlying topology more apparent. We assume the reader is familiar with the language of operads and $\EE_{\infty}$-spaces, and the reader unfamiliar with these notions is referred to \cite{may_geometry_1972}.  For the remainder of the section, we fix an $\EE_{\infty}$-operad $\cO$, given by the geometric realization of the Barratt--Eccles operad \cite{May74}. When we say $\EE_{\infty}$-space, or connective spectrum, we mean an $\cO$-algebra. We use $\Sp^{\geq 0}$ to denote a $1$-category of connective spectra, which may be modeled, for instance, by sequential spectra or grouplike $\mathbb{E}_{\infty}$-spaces.

We briefly summarize Thomason's argument from \cite{thomason_symmetric_1995}. The main construction is a functor\[
\Null/\colon \Sp^{\geq 0} \to \LaxSym
\] which associates every connective spectrum $E$ to a lax symmetric monoidal category $\Null/\Loop^{\infty}E$, recalled below in \cref{def:null X}. In \cite[\S 5]{thomason_symmetric_1995}, Thomason proves that the composite\[
\Sp^{\geq 0} \xrightarrow{\Null/} \LaxSym \xrightarrow{S} \Sym
\] is an inverse equivalence (after localization at stable equivalences) to $K^{S}$. 

A lax symmetric monoidal category is like a symmetric monoidal category, i.e.\ a category $\cat C$ equipped with an $n$-fold tensor product $T^n\colon \cat C^{\times n}\to \cat C$ for $n\geq 0$ (with $*\to \cat C$ picking out a distinguished object), except the associator and unitor transformations need not be invertible; see \cite[Definition 3.1.1]{leinster_higher_2004} for a complete description. 
Lax symmetric monoidal categories assemble into a $2$-category $\LaxSym$ whose $1$-cells are given by lax symmetric monoidal functors and whose $2$-cells are lax symmetric monoidal natural transformations (see \cite[Definitions 3.1.3 and 3.1.4]{leinster_higher_2004}). The inclusion of symmetric monoidal categories into lax ones gives a $2$-functor $U\colon \Sym\to \LaxSym$.

In \cite[Proposition 2.5]{thomason_symmetric_1995}, Thomason constructs a $2$-functor $
S\colon \LaxSym\to \Sym$ along with a 2-categorical natural transformation $\eta\colon S\circ U\Rightarrow 1_{\Sym}$. Additionally, he extends Segal's group completion functor $K^{S}\colon \Sym\to \Sp^{\geq 0}$ to apply to lax symmetric monoidal categories. In particular, there is a functor\[
Spt\colon \LaxSym\to \Sp^{\geq 0}
\] which, for $\cat C\in \LaxSym$, produces the group completion of the classifying space $B\cat C$. He produces a natural chain of weak equivalences between $Spt(\cat C)$ and $Spt(US\cat C)$, and for any symmetric monoidal category $\cat A$, there is a natural chain of weak equivalences between $Spt(U\cat A)$ and $K^{S}(\cat A)$. Consequently, for any $\cat C\in \LaxSym$, there is a natural chain of weak equivalences between $Spt(\cat C)$ and $K^{S}(S\cat C)$.

The missing piece of the argument is then to associate a connective spectrum $E$ to a lax symmetric monoidal category whose group completion $K$-theory recovers $E$. This is precisely the role fulfilled by $\Null/\Loop^{\infty} E$, whose definition we now recall.

\begin{definition}\label{def:null X}
    Let $X$ be an $\EE_{\infty}$-space and define $\Null/X$ to be the category whose objects are continuous maps $f_C\colon C\to X$, where $C$ is a weakly contractible space. We will abbreviate objects as simply $C$.  A morphism $g\colon C\to D$ is simply a continuous map such that $f_{D}\circ g = f_{C}$.
\end{definition}

In \cite[(4.5.2--4.5.5)]{thomason_symmetric_1995}, Thomason specifies a lax symmetric monoidal structure on $\Null/X$ by using the $\mathbb{E}_\infty$-structure on $X$. For $n\geq 1$, the $n$-fold tensor product $T^n(C_1, \dots, C_n)$ is given by\[
\mathcal O(n) \times C_1\times \dots \times C_n \xrightarrow{1\times f_1\times\dots \times f_{n}} \mathcal O(n) \times X \times \dots \times X \xrightarrow{\alpha_n} X,
\] where $\alpha_n$ is the $\mathcal{O}$-algebra structure map for $X$.  The map $\iota_C\colon C\to T^1(C)$ is the map 
$
    C\to \cO(1)\times C
$
given by $c\mapsto (\mathrm{id},c)$ where $\mathrm{id}\in \cO(1)$ is the distinguished unit element.  The associativity and symmetry relations come for free from the corresponding relations for algebras over an operad. The consequence of Thomason's work is that there is a natural zig-zag of weak equivalences $X\leftrightarrow \Loop^{\infty} K^S (S(\Null/X))$ when $X$ is grouplike, and therefore by \cref{thm:main} there is a natural zig-zag of weak equivalences $X\leftrightarrow \Loop^{\infty} K^W(\Gamma S(\Null/X))$.

Our goal is to construct a Waldhausen category $\cat N(X)$ which produces an equivalent $K$-theory spectrum to $\Gamma S(\Null/X)$, in the spirit of the $\Gamma$-construction but more evidently related to $\Null/X$ and more explicit. In particular, the objects of $\cat N(X)$ are ``homotopically discrete'' retractive spaces over $X$, but the morphisms are quite different than just maps rel $X$.

\begin{definition}\label{definition: N(X)}
    Let $\cat N(X)$ be the category whose objects are retractive spaces 
    \[
        \begin{tikzcd}
            A\ar[r,shift left, "r_A"] & X\ar[l,shift left,"s_A"]
        \end{tikzcd}
    \]
    over $X$ of the form $A = X\amalg \coprod_{i=1}^a A_i$, where each $A_i$ is weakly contractible and the inclusion $X\to A$ is the identity map from $X$ to the first summand.
    A morphism $(\phi, \vec{f})\colon A\to B$ is given by the following data:\begin{itemize}
    \item an isomorphism class of spans $[\phi\colon \underline{b}\leftarrow M \to \underline{a}]$ of finite sets;
    \item for $1\leq i\leq a$, with $\phi(i) = \{j_1\leq \dots\leq j_p\}$, a map $f_i\colon A_i\to T^p(B_{j_1}, \dots, B_{j_p})$ such that \[
    \begin{tikzcd}
        A_i \ar[r, "f_i"] \ar[d, swap, "r_A"]& T^{p}(B_{j_1}, \dots, B_{j_{p}}) \ar[d, "{T^{p}(r_{B}, \dots, r_B)}"] \\
        X & T^{p}(X, \dots, X) \ar[l, "\alpha_{p}"]
    \end{tikzcd}
    \] commutes.  We adopt the notation $T^{\phi(i)}(B_{\phi(i)}) := T^p(B_{j_1}, \dots, B_{j_p})$.  Note that when $\phi(i)=\emptyset$ the map $f_i\colon A_i\to T^0() = *$ is necessarily the terminal map.
\end{itemize}
\end{definition}

Since every space in sight (except $X$) is weakly contractible, the maps $f_i$ are all weak homotopy equivalences.  Thus we think about the data of a map $A\to B$ as an explicit choice of how to ``replace'' some of the $B_i$ with some of the $A_i$, at least up to weak equivalence. The $j\in \underline{b}$ so that $j\not\in \phi(i)$ for any $i$ correspond to the components $B_j$ that receive \textit{no map} from any $A_i$, and similarly the $i\in \underline{a}$ with $\phi(i) = \varnothing$ correspond to the $A_i$ which are not mapped to any $B_j$. 

The idea essentially is that we can extend the definition of $\Gamma$ to take lax symmetric monoidal categories as input; we should think of an object $A\in \cat N(X)$ as defining the tuple $(A_i)_{i=1}^n$ in $\Gamma(\Null/X)$. All of the definitions (identities, composition, cofibrations, weak equivalences, and pushouts) work essentially the same way as in $\Gamma(\cat C)$, except they are slightly more complicated due to the lax structure. We detail how these definitions work for $\cat N(X)\cong \Gamma(\Null/X)$, although it is not difficult to see how one would extend the definition for arbitrary lax symmetric monoidal categories.

\begin{definition}
The identity maps in $\cat N(X)$ are given by $(\id_{\underline{a}},\vec{f})$ where each $f_i\colon A_i\to T^1(A_i)=A_i\times \mathcal{O}(1)$ sends $x\in A_i$ to $(x,\id)$ where $\id$ is the unit of $\mathcal{O}$. Composition is defined just as in $\Gamma$, except that checking that the result is actually another morphism in $\cat N(X)$ becomes a bit more involved (one needs to use the axioms of lax symmetric monoidal categories, the associative law for algebras over operads, and the symmetric equivariance of the operad action).
\end{definition}

\begin{proposition}
    There is a Waldhausen structure on $\cat N(X)$ given by:\begin{itemize}
        \item the zero object is $X$ itself;
        \item the weak equivalences are as in \cref{defintion: cofibs and wes};
        \item the cofibrations are as in \cref{defintion: cofibs and wes}, except that the maps $A_i\to T^1(B_{\phi(i)})$ are required to factor as\[
        A_i \xrightarrow{f_i} B_{\phi(i)} \xrightarrow{(-,\id)} B_{\phi(i)}\times \mathcal O(1)
        \] where $f_i$ is a homeomorphism;
        %\item pushouts are as in \cref{defn: pushout}, except that (for $j\not\in \im\phi$) $k_j=(-,\id)\colon B_j\to T^1(B_j)$.  
    \end{itemize}
\end{proposition}

Essentially the same proofs as the previous section work to show this indeed defines a Waldhausen structure, with the only additional complexity coming from keeping track of the additional coherence data in a lax monoidal category. Moreover, $\cat N(X)$ has a choice of wedges as before, with\[
A \cup_X B := X \amalg \coprod_{i=0}^a A_i \amalg \coprod_{j=0}^b B_j.
\] The same proof as \cref{prop:weakly split cof} show $\cat N(X)$ has weakly split cofibrations, and consequently $K^S(w\cat N(X))\xrightarrow{\sim} K^W(\cat N(X))$. 
Finally, the analogue of the inclusion $\cat C\hookrightarrow\Gamma(\cat C)$ is given by an oplax symmetric monoidal functor $s\colon \Null/X\to w\cat N(X)$ which, on objects, is given by $s(C) = X\amalg C$, viewed as a retractive space using the given map $C\to X$. A morphism $h\colon C\to D$ is sent to the map \[
X\amalg C \xrightarrow{(\id_{\underline{1}}, h')} X\amalg D
\] supplied by the data of the identity on $X$ and $C\xrightarrow{h} D \xrightarrow{\iota_D} D\times \mathcal{O}(1)$. 

This definition is evidently functorial, and we can check that it is oplax symmetric monoidal
\[
    s(T^n(C_1,\dots, C_n)) = X\amalg (C_1\times \dots \times C_n\times O(n)) \xrightarrow{(\phi,f)} X\amalg \coprod_{i=1}^n C_i  = s(C_1) \cup_X\dots \cup_X s(C_n)
\] 
where 
\[
(\phi,f) = (\underline{n}\xleftarrow{=}\underline{n}\xrightarrow{!}\underline{1},\id_{T^n(C_1,\dots, C_n)}).
\]
The definition on morphisms is similar. However, as before, this inclusion is not lax unital: the lax unit in $\Null/X$ is $x_0\to X$, where $x_0\in X$ is the image of the map $\mathcal O(0)\to X$, whereas the unit in $\cat N(X)$ is $X$. The same fix as before works, using the auxiliary category $(\Null/X)_+$, factoring the inclusion as $\Null/X\hookrightarrow (\Null/X)_+ \to w\cat N(X)$, and showing that the first map induces an equivalence after group completion and the second induces an equivalence on classifying spaces.

\begin{corollary}
    There is a functor $\cat N\colon \EE_{\infty}\text{-spaces}\to \Wald_{\vee}$ such that $K^W(\cat N(X))\simeq K^{S}(w\cat N(X))$ for all $X\in \Sp^{\geq 0}$.
\end{corollary}\begin{proof}
    For all $\mathbb{E}_\infty$-spaces $X$, the category $\cat N(X)$ is a Waldhausen category and given $f\colon X\to Y$, we obtain an exact functor $\cat N(f)\colon \cat N(X)\to \cat N(Y)$ which is defined by sending $X\amalg \coprod_{i=1}^a A_i$ to $Y\amalg \coprod_{i=1}^a A_i$, where $A_i$ is a space over $Y$ via the composition $A_i\to X\xrightarrow{f} Y$. Since $f$ is a map of $\mathbb{E}_\infty$-spaces, the evident definition of $\cat N(f)$ on morphisms also makes sense and one checks that $\cat N(f)$ is exact and $\cat N(-)$ is functorial. The second claim follows from \cref{theorem: weakly split K theory is the same} and that the cofibrations in $\cat N(X)$ are splittable up to weak equivalence.
\end{proof}

We can now complete the proof of \cref{introthm:inverse K} as stated in the introduction. 

\begin{proof}[Proof of \cref{introthm:inverse K}]
Let $E$ be a connective spectrum and set $X= \Loop^{\infty} E$. By the corollary above, there is a natural weak equivalence $K^{S}(w\cat N(X))\xrightarrow{\sim} K^W(\cat N(X))$. On the other hand, the functor $s\colon \Null/X \to w\cat N(X)$ will induce an equivalence $Spt(Uw\cat N(X))\to Spt(\Null/X)$. By Thomason's theorem, we then have a natural zig-zag of weak equivalences\[
\Loop^{\infty}K^W(\cat N(X)) \xleftarrow{\sim} \Loop^{\infty} K^{S}(w\cat N(X)) \leftrightarrow Spt(Uw\cat N(X)) \xleftarrow{\sim} Spt(\Null/X) \leftrightarrow X. 
\] Therefore $E$ is isomorphic to $K^W(\cat N(\Loop^\infty E))$ in the stable homotopy category. %This shows that $K^W$ is essentially surjective, and the localization claim follows immediately from the definition of the morphisms $W_K$ as exact functors which induce stable equivalences on $K$-theory spectra.
\end{proof}

%---------------------------------
%---------------------------------
\printbibliography
\end{document}